\documentclass[11pt]{article}
\usepackage{amsmath}
\usepackage{dsfont}
\usepackage{mathrsfs}
\usepackage{amsmath,amssymb}
\usepackage{amsfonts}
\usepackage{hyperref}
\usepackage{amsthm}
\usepackage{graphicx}
\usepackage{subfigure}
\usepackage{xcolor}

\hfuzz=\maxdimen
\theoremstyle{definition}
\newtheorem{lemma}{Lemma}[section]

\newtheorem{theorem}[lemma]{Theorem}
\newtheorem{corollary}[lemma]{Corollary}

\newtheorem{remark}{Remark}

\numberwithin{equation}{section}

\DeclareFixedFont{\Acknowledgment}{OT1}{cmr}{bx}{n}{14pt}
\begin{document}

\title{\bf Gradient flow of the norm squared of a moment map over K$\ddot{\textmd{a}}$hler manifolds}
\author{Aijin Lin and Liangming Shen}
\date{}
\maketitle

\begin{abstract}
Inspired by Wilkin's work \cite{VW, W} on Morse theory for the moduli space of Higgs bundles, we study the moduli space of gauged holomorphic maps by a heat flow approach in the spirit of Atiyah and Bott in a series of papers. In this paper, applying the method of Hong \cite{H}, we establish the global existence of smooth solutions of the gradient flow equations of the vortex functional over a compact K$\ddot{\textmd{a}}$hler manifold.
\end{abstract}
\vspace{12pt}
\let\thefootnote\relax \footnotetext {The first author is supported by the National Natural Science Foundation of China
(Grant No. 11401578). The second author is supported by the National Science Foundation under Grant No. DMS-1440140. }

%\textbf{Mathematics Subject Classification (2010).} 52C25, 52C26, 53C44.\\
%\tableofcontents

\section{Introduction}
The Yang-Mills heat flow-the $L^{2}$-gradient flow for the Yang-Mills functional
 $$ A\mapsto ||F_{A}||_{L^{2}}$$
 was first introduced by Atiyah and Bott \cite{AB}, where $A$ is a connection on a principal bundle over a Riemann surface. They used the Yang-Mills heat flow to study Morse theory of the Yang-Mills functional over Riemann surfaces and proved that there was a Morse stratification coming from the Yang-Mills heat flow coincides with a holomorphic stratification coming from the Narasimhan-Seshadri type of the holomorphic structures on the vector bundle,  assuming that the flow has sufficiently good analytical properties. Donaldson \cite{D} used the Yang-Mills heat flow to give a new proof of and generalize a theorem by Narasimhan and Seshadri on stable holomorphic bundles. Daskalopoulos  \cite{Da} and Rade \cite{R} proved that holomorphic and Morse stratifications agree by different approaches. Similar Morse theory for the space of holomorphic vector bundles equipped with some extra data has been studied. For example, Wilkin \cite{W} studied the space of Higgs pairs $(A, \phi)$ by using the gradient flow of the Yang-Mills type functional
  $$(A, \phi)\mapsto ||F_{A}+[\phi, \phi^{*}]||_{L^{2}},$$
where $A$ is a connection on a principal bundle over a Riemann surface, $\phi$ is a $E$-valued $(0, 1)$ form such that $\bar{\partial}_{A}\phi=0$.  Wilkin \cite{DWWW} used Morse theoretic techniques to compute the equivariant Betti numbers of the space of semistable rank two degree zero Higgs bundles over a compact Riemann surface in the spirit of Atiyah and Bott's original approach. Recently, Venugopalan \cite{V} studied the gradient flow of the vortex functional
$$(A, u)\mapsto ||F_{A}+\Phi(u)||_{L^{2}},$$
 where $A$ is a connection on a principle bundle over a compact Riemann surface and $u$ is a holomorphic section of the associated fiber bundle whose fiber is a K$\ddot{\textmd{a}}$hler manifold with a Hamiltonian group action. Venugopalan \cite{VW} used the gradient flow of the vortex functional to study the classification of affine vortices.\par
 Let $X$ be a compact K$\ddot{\textmd{a}}$hler manifold. Fix a K$\ddot{\textmd{a}}$hler metric on $X$ and let $\omega$  be the associated K$\ddot{\textmd{a}}$hler form. Let $E$ be a rank $n$ holomorphic vector bundle over $X$. We will consider the vector bundle $E$ to be endowed with a fixed K$\ddot{\textmd{a}}$hler metric $H_{0}$. Let $\Omega^{0}(X, E)$ denote the smooth sections of $E$ and let $\mathcal{A}$ denote the space of all connections on $E$ that are unitary with respect to $H_{0}$. Let $\mathcal{G}=\textmd{Aut}(E)$ denote the gauge group, which acts on $\mathcal{A}\times \Omega^{0}(X, E)$.
We will also need to consider the spaces $\Omega^{p, q}(X, E)$ and $\Omega^{p, q}(X, \textmd{End} E)$, i.e. the spaces of forms of holomorphic type $(p, q)$ with values in $E$ and in the endomorphism bundle of $E$ respectively. Using the metric on $E$ we get identifications $E\approx E^{*}$ and also $E\otimes E^{*}\approx \textmd{End} E$.
 Define the $\textit{Yang-Mills-Higgs functional}$
$$\mathcal{YMH}: \mathcal{A}\times \Omega^{0}(X, E)\rightarrow \mathbb{R}$$
by
\begin{equation} \label{YMH}
\mathcal{YMH}(A, \phi)=||F_{A}||_{L^{2}}^{2}+||d_{A}\phi||_{L^{2}}^{2}+\frac{1}{4}||\phi\otimes\phi^{*}-\tau I||_{L^{2}}^{2}.
\end{equation}
Here $F_{A}\in \Omega^{2}(X, \textmd{End} E)$ is the curvature of the connection $A$, and $d_{A}\phi\in \Omega^{1}(X, E)$ is the covariant derivative of the section $\phi$, $I\in \Omega^{0}(X, \textmd{End} E)\approx \Omega^{2}(X, E\otimes E^{*})$ is the identity section and $\tau$ is a real parameter. The adjoint $\phi^{*}$ of $\phi$ is taken with respect to $H_{0}$.\par
The Yang-Mills-Higgs functional \ref{YMH}, which generalizes the Yang-Mills functional, appeared in the theory of superconductivity. Especially, $\mathcal{YMH}(A, \phi)$ measures the thermodynamic free energy and the physical configurations are (symplectic) vortices which minimize the functional. It attracted a lot of interests \cite{CGS, JT, I1, I2, IT, SW, T, Tr, Y} among physicists and mathematicians during the past decades. \par
Let $d_{A}: \Omega^{0}(X, E)\rightarrow \Omega^{1}(X, E)$  be the covariant derivative with respect to the connection $A$. We can split
$$d_{A}=\partial_{A}+\bar{\partial}_{A}$$
due to the splitting $$\Omega^{1}(X, E)=\Omega^{0, 1}(X, E)\oplus \Omega^{1, 0}(X, E)$$
coming from the complex structure $J$ on $X$. Now we can rewrite the Yang-Mills-Higgs functional $\mathcal{YMH}$ \ref{YMH} by the following energy identity in \cite{B}.
 \begin{theorem}\label{EN}
 Let $c_{1}(E)\in H^{2}(X, \mathbb{R})$ and $ch_{2}(E)\in H^{4}(X, \mathbb{R})$ be the first Chern class and second Chern character of $E$ respectively. Then the Yang-Mills-Higgs functional $$\mathcal{YMH}: \mathcal{A}\times \Omega^{0}(X, E)\rightarrow \mathbb{R}$$
can be written as
  \begin{eqnarray*}
\mathcal{YMH}(A, \phi)= \|\Lambda F_{A}-\frac{i}{2}(\phi\otimes\phi^{*}-\tau I)\|_{L^{2}}^{2}+4\|F_{A}^{0,2}\|_{L^{2}}^{2}\\
+2\|\bar{\partial}_{A}\phi\|+2\pi\tau C_{1}(E, \omega)-8\pi^{2}Ch_{2}(E, \omega),
 \end{eqnarray*}
 where $2\pi\tau C_{1}(E, \omega)-8\pi^{2}Ch_{2}(E, \omega)$ is a topological invariant,  $\Lambda F_{A}=(F_{A}, \omega)$, $( , \omega)$ denotes pointwise inner product with the K$\ddot{\textmd{a}}$hler form $\omega$, and  $$C_{1}(E, \omega)=:\int_{X}c_{1}(E)\wedge\omega^{[n-1]}, \quad  Ch_{2}(E, \omega)=:\int_{X}ch_{2}(E)\wedge\omega^{[n-2]}.$$
 \end{theorem}
The energy identity above immediately implies that the functional $\mathcal{YMH}$ attains its lower bound $2\pi\tau C_{1}(E, \omega)-8\pi^{2}Ch_{2}(E, \omega)$ if and only if $(A, \phi)\in  \mathcal{A}\times \Omega^{0}(X, E)$ satisfies the following (symplectic) $\textit{vortex equations}$
\begin{eqnarray}\label{vortex}
\begin{split}
F_{A}^{0, 2}=0, \quad
\bar{\partial}_{A}\phi=0,\\
\Lambda F_{A}-\frac{i}{2}(\phi\otimes\phi^{*}-\tau I)&=0.
\end{split}
\end{eqnarray}

In order to find the absolute minima characterized by the vortex equations \ref{vortex},  first define the space $\mathcal{U}_{\tau}$ of $\tau$-$Vortex$ pairs on $E$ as the space of solutions of the vortex equations above, i.e.
\begin{equation}
\mathcal{U}_{\tau}=\{(A, \phi)\in\mathcal{A}\times \Omega^{0}(X, E)|\mathcal{YMH}(A, \phi)=2\pi\tau C_{1}(E, \omega)-8\pi^{2}Ch_{2}(E, \omega)\}.
\end{equation}
Consider only integrable unitary connections, i.e. those which belong to $\mathcal{A}^{1,1}$, where
\begin{equation}
\mathcal{A}^{1,1}=\{A\in\mathcal{A}^{1,1}|F_{A}^{2,0}=F_{A}^{0,2}=0\}.
\end{equation}
Then we look for minimizing pairs $(A, \phi)$ in $\mathcal{H}\subset \mathcal{A}^{1, 1}\times \Omega^{0}(X, E)$, where
\begin{equation}
\mathcal{H}=\{(A, \phi) \in \mathcal{A}^{1, 1}\times \Omega^{0}(X, E)|\bar{\partial}_{A}\phi=0\},
\end{equation}
and any a pair $(A, \phi)\in \mathcal{H}$ is called a $\textit{gauged holomorphic map }$.
\begin{remark}\label{rem1}
 Similar definition can be found in \cite{V}, where the base manifold is a compact Riemann surface and the fiber of the associated fiber bundle is a compact K$\ddot{\textmd{a}}$hler manifold with a Hamiltonian group action. Thus our definition of the gauged holomorphic map can be regarded as a generalization of Venugopalan's when the base manifold is a compact K$\ddot{\textmd{a}}$hler manifold and the fiber of the associated fiber bundle is $\mathbb{C}^{n}$.
\end{remark}
It is well-known that the space $\mathcal{H}$ is a symplectic submanifold of $\mathcal{A}\times \Omega^{0}(X, E)$ with a Hamiltonian action of the group of gauge transformations $\mathcal{G}$. The moment map $\Psi_{\tau}: \mathcal{H}\rightarrow \mathfrak{g}$ is given by (see \cite{BD})
\begin{equation}\label{moment}
\Psi_{\tau}(A, \phi)=\Lambda F_{A}-\frac{i}{2}(\phi\otimes\phi^{*}-\tau I).
\end{equation}
 Now we define the $\textit{vortex functional}$ (see section 2 for details) on the space of gauged holomorphic maps
$$\mathcal{F}: \mathcal{H}\rightarrow \mathbb{R}$$
by
 \begin{equation}\label{vort}
 (A, \phi)\mapsto \mathcal{F}(A, \phi)=\|\Lambda F_{A}-\frac{i}{2}(\phi\otimes\phi^{*}-\tau I)\|_{L^{2}}^{2},
 \end{equation}
 which implies that the vortex functional is precisely the square of the norm of the moment map \ref{moment}.\par
  By discussion above, we have
  $$\mathcal{U}_{\tau}=\{(A, \phi)\in \mathcal{H}|\Lambda F_{A}-\frac{i}{2}(\phi\otimes\phi^{*}-\tau I)=0\}.$$
Note that the Yang-Mills-Higgs functional $\mathcal{YMH}$ is invariant under the action of the unitary gauge group $\mathcal{G}$. It means that $\mathcal{U}_{\tau}$ is $\mathcal{G}$-invariant set and that $\tau$-vortices can be defined by equivalence classes in $\mathcal{B}=(\mathcal{A}\times \Omega^{0}(X, E))/\mathcal{G}$. Define the space of gauge equivalence classes of $\tau$-$Vortex$ pairs by
\begin{equation}
\mathcal{V}_{\tau}=\{[A, \phi]\in\mathcal{B}|\mathcal{YMH}(A, \phi)=2\pi\tau C_{1}(E, \omega)-8\pi^{2}Ch_{2}(E, \omega)\}.
\end{equation}
Bradlow has studied the moduli space $\mathcal{V}_{\tau}$ systematically. In \cite{B}, by applying the famous results of Kazdan and Warner \cite{KW}, Bradlow obtained a complete description of the moduli space $\mathcal{V}_{\tau}$ in terms of a class of divisors in the base manifold $X$ when $X$ is closed K$\ddot{\textmd{a}}$hler manifold and the rank of the vector bundle $E$ is one. When the base manifold $X$ is a closed Riemann surface, Bradlow \cite{BD} proved that the moduli space $\mathcal{V}_{\tau}$ is a finite dimensional Hausdorff compact K$\ddot{\textmd{a}}$hler manifold under suitable restrictions on $\tau$. Further, he \cite{BD2} proved that the moduli space $\mathcal{V}_{\tau}$ can admit the structure of a nonsingular projective variety and obtained some topological information on $\mathcal{V}_{\tau}$. \par
 Inspired by Wilkin's work \cite{VW, W} on Morse theory for the space of Higgs bundles, we will study the moduli space $\mathcal{V}_{\tau}$ systematically by the following gradient flow equations of the vortex functional \ref{vort} in the spirit of Atiyah and Bott in a series of papers.
\begin{eqnarray}\label{f1}
\frac{\partial A}{\partial t}=-*d_{A}F_{A, \phi},  \quad
\frac{\partial \phi}{\partial t}=-JL_{\phi}F_{A, \phi},
\end{eqnarray}
where $L_{x}: \mathfrak{g} \rightarrow T_{x}X$ denotes the infinitesimal action of $\mathfrak{g}$ on $X$, $\mathfrak{g}$ is the Lie algebra of the gauge group $\mathcal{G}$, $J$ is the fixed complex structure of the K$\ddot{\textmd{a}}$hler manifold $X$, and $F_{A, \phi}=\Lambda F_{A}-\frac{i}{2}(\phi\otimes\phi^{*}-\tau I)$.\par
In this paper we mainly focus on the global existence of smooth solutions of the gradient flow equations \ref{f1}. In next papers \cite{L1, L2}, we will study the convergence of the gradient flow equations \ref{f1} and start to construct the equivariant Morse theory of the vortex functional \ref{vort} in the spirit of Atiyah and Bott.\par
 Similar equations compared to \ref{f1} over a compact Riemann surface have been studied, for example in \cite{V, W}. However, the base manifold here is a compact K$\ddot{\textmd{a}}$hler manifold so we need to apply different approach. The key step comes from the energy identity (see section 2) which associates the vortex functional \ref{vort} with the Yang-Mills-Higgs functional \ref{YMH}. In fact, we can show that the vortex functional is equivalent to the Yang-Mills-Higgs functional in the space $\mathcal{H}$ of gauged holomorphic maps by the energy identity. Therefore, the gradient flow equations \ref{f1} is equivalent to the gradient flow equations of the Yang-Mills-Higgs functional
  \begin{equation}\label{f2}
\begin{split}
\frac{\partial A}{\partial t}&=-d^{*}_{A}F_{A}-J_{A,\phi}; \\
\frac{\partial \phi}{\partial t}&=-d^{*}_{A}d_{A}\phi+\frac{1}{2}\phi(\tau-|\phi|^{2}).
\end{split}
\end{equation}
Now we state our main results as follows.\par
First by applying and generalizing Hong's approach \cite{H}, we show that there exists a global smooth solutions of the gradient flow equations \ref{f2}. In fact, we have
 \begin{theorem}\label{Com}
Then the gradient flow equations \ref{f2}
 \begin{equation*}
\begin{split}
\frac{\partial A}{\partial t}&=-d^{*}_{A}F_{A}-J_{A,\phi}; \\
\frac{\partial \phi}{\partial t}&=-d^{*}_{A}d_{A}\phi+\frac{1}{2}\phi(\tau-|\phi|^{2}).
\end{split}
\end{equation*}
of the Yang-Mills-Higgs functional over a compact K$\ddot{\textmd{a}}$hler manifold with initial conditions $(A_{0}, \phi_{0})\in \mathcal{H}$, where $J_{A,\phi}=\frac{1}{2}(d_{A}\phi\otimes\phi-\phi\otimes(d_{A}\phi)^{*})$,  have a unique smooth solution which exists for all time and depends continuously on the initial conditions.
  \end{theorem}
Then by the equivalence (See Theorem \ref{EN}) between gradient flow equations \ref{f1} and \ref{f2}, we get
 \begin{theorem}\label{Exi}
The gradient flow equations \ref{f1}
 \begin{eqnarray*}
\frac{\partial A}{\partial t}=-*d_{A}F_{A, \phi},  \quad
\frac{\partial \phi}{\partial t}=-JL_{\phi}F_{A, \phi}.
\end{eqnarray*}
of the vortex functional \ref{vort} over a compact K$\ddot{\textmd{a}}$hler manifold  with initial conditions $(A_{0}, \phi_{0})\in \mathcal{H}$, where $F_{A, \phi}=\Lambda F_{A}-\frac{i}{2}(\phi\otimes\phi^{*}-\tau I)$, have a unique smooth solution which exists for all time and depends continuously on the initial conditions.
  \end{theorem}
As an application, suppose that the base manifold $X$ is one dimensional compact K$\ddot{\textmd{a}}$hler manifold, which means $X$ is a compact Riemann surface.
Note that in this case $\mathcal{A}=\mathcal{A}^{1, 1}$, and
\begin{equation*}
\mathcal{H}=\{(A, \phi) \in \mathcal{A}\times \Omega^{0}(X, E)|\bar{\partial}_{A}\phi=0\}.
\end{equation*}
Therefore we have
\begin{corollary}\label{Exi1}
The gradient flow equations
 \begin{eqnarray*}
\frac{\partial A}{\partial t}=-*d_{A}F_{A, \phi},  \quad
\frac{\partial \phi}{\partial t}=-JL_{\phi}F_{A, \phi}.
\end{eqnarray*}
of the vortex functional \ref{vort} over a compact Riemann surface with initial conditions $(A_{0}, \phi_{0})\in \mathcal{H}$, where $F_{A, \phi}=\Lambda F_{A}-\frac{i}{2}(\phi\otimes\phi^{*}-\tau I)$,  have a unique smooth solution which exists for all time and depends continuously on the initial conditions.
  \end{corollary}
This corollary can be regarded as a special case of Theorem 1.5 in \cite{V}, when the fiber of the associated fiber bundles is a special symplectic vector space, i.e. $\mathbb{C}^{n}$. \par
This paper is organized as follows. Section 2 gives some preliminary analysis used in the rest of the paper. In section 3 we prove the local existence of smooth solutions of the gradient flow equations \ref{f2}. We complete the proofs of Theorem \ref{Com} and Theorem \ref{Exi} in section 4.

\noindent \textbf{Acknowledgements:} Both authors would like to thank Professor Gang Tian for constant guidance and encouragement. The first author also thanks Professor Huijun Fan for constant guidance and support. The first author
is supported by the National Natural Science Foundation of China
under Grant No. 11401578. The second author is supported by the National Science Foundation under Grant No. DMS-1440140.

\section{Preliminary analysis}
The purpose of this section is to give some preliminary analysis used in the rest of the paper. First observe that if we constran the the Yang-Mills-Higgs functional $\mathcal{YMH}$ on the space $\mathcal{H}$ of gauged holomorphic maps, then by Theorem \ref{EN} we have
\begin{eqnarray}
\begin{split}\label{identity}
\mathcal{YMH}(A, \phi)&=\mathcal{F}(A, \phi)+2\pi\tau C_{1}(E, \omega)-8\pi^{2}Ch_{2}(E, \omega)\}\\
&=\|\Lambda F_{A}-\frac{i}{2}(\phi\otimes\phi^{*}-\tau I)\|_{L^{2}}^{2}+2\pi\tau C_{1}(E, \omega)-8\pi^{2}Ch_{2}(E, \omega).
\end{split}
\end{eqnarray}
Recall that $C=2\pi\tau C_{1}(E, \omega)-8\pi^{2}Ch_{2}(E, \omega)$ is a topological constant, therefore we obtain
\begin{theorem}\label{equiv}
In the space $\mathcal{H}$ of gauged holomorphic maps, the gradient flow equations \ref{f1}
\begin{eqnarray*}
\frac{\partial A}{\partial t}=-*d_{A}F_{A, \phi},  \quad
\frac{\partial \phi}{\partial t}=-JL_{\phi}F_{A, \phi}.
\end{eqnarray*}
of the vortex functional \ref{vort} are equivalent to the gradient flow equations \ref{f2}
 \begin{equation}\label{heat2}
\begin{split}
\frac{\partial A}{\partial t}&=-d^{*}_{A}F_{A}-J_{A,\phi}; \\
\frac{\partial \phi}{\partial t}&=-d^{*}_{A}d_{A}\phi+\frac{1}{2}\phi(\tau-|\phi|^{2}).
\end{split}
\end{equation}
of the Yang-Mills-Higgs functional \ref{YMH}.
\end{theorem}
\begin{proof}
The Euler-Lagrange equations of the vortex functional can be computed as follows
$$\nabla\mathcal{F}(A, \phi)=(*d_{A}F_{A, \phi},\quad JL_{\phi}F_{A, \phi})=0.$$
We can also compute the the Euler-Lagrange equations of the Yang-Mills-Higgs functional
as
$$\nabla\mathcal{YMH}(A, \phi)=(d^{*}_{A}F_{A}+J_{A,\phi},\quad d^{*}_{A}d_{A}\phi-\frac{1}{2}\phi(\tau-|\phi|^{2})=0.$$
By the energy identity \ref{identity} in $\mathcal{H}$, note that $C=2\pi\tau C_{1}(E, \omega)-8\pi^{2}Ch_{2}(E, \omega)$ is a topological constant independent of $A, \phi$,
thus the Euler-Lagrange equations above agree.  $\hfill\Box$
\end{proof}
Next we set some notations and do some computations. For the K$\ddot{\textmd{a}}$hler manifold $X$ with K$\ddot{\textmd{a}}$hler form $\omega$ we define
\begin{eqnarray*} L: \Omega(X, E)^{p,q}\rightarrow \Omega(X, E)^{p+1,q+1}, L(\eta)=\eta\wedge\omega. \end{eqnarray*}
We can define an algebraic trace operator $\Lambda$ by
\begin{eqnarray*} \Lambda=L^{*}:\Omega(X, E)^{p,q}\rightarrow \Omega(X, E)^{p-1,q-1}. \end{eqnarray*}
Then the operators ${\partial}_{A}$, ${\bar \partial}_{A}$, their $L^{2}$ adjoints and $\Lambda$ are related by the following K$\ddot{\textmd{a}}$hler identities (See \cite{H, J} for details)
\begin{eqnarray}\label{Ka1}
{\bar \partial}^{*}_{A}=i[\partial_{A}, \Lambda],\quad
{\partial}^{*}_{A}=-i[{\bar \partial}_{A}, \Lambda]
\end{eqnarray}
on $\Omega^{p, q}(X, E)$. Especially, we have
\begin{eqnarray}\label{Ka2}
{\bar \partial}^{*}_{A}=-i\Lambda\partial_{A},\quad
{\partial}^{*}_{A}=i\Lambda{\bar \partial}_{A}
\end{eqnarray}
on $\Omega^{0, 1}(X, E), \Omega^{1, 0}(X, E)$.\par
Following Donaldson's approach \cite{D, H}, consider the complex gauge group ${\mathcal{G}}^{\mathbb{C}}$, which acts on $A\in \mathcal{A}^{1,1}$ with curvature $F_{A}$ of type $(1,1)$ by
\begin{eqnarray} \label{G1}
\bar{\partial}_{g(A)}=g\cdot \bar{\partial}_{A}\cdot g^{-1}, \quad {\partial}_{g(A)}={g^{*}}^{-1} \cdot {\partial}_{A}\cdot g^{*}
\end{eqnarray}
where $g^{*}=\bar{g}^{t}$ denotes the conjugate transpose of $g$. Extending the action of the unitary gauge group,
\begin{eqnarray*}
\mathcal{G}=\{g\in {\mathcal{G}}^{\mathbb{C}}|h(g)=g^{*}g=I\},\end{eqnarray*}
which means
\begin{eqnarray} \label{G2}
g^{-1}\cdot d_{g(A)}\cdot g=\bar{\partial}_{A}+h^{-1}{\partial}_{A}h,\quad
g^{-1}F_{g(A)}g=F_{A}+\bar{\partial}_{A}(h^{-1}{\partial}_{A}h),
\end{eqnarray}
where $h=g^{*}g$.\par

Suppose that $(A_{0}, \phi_{0})\in \mathcal{H}$, which means $A_{0}\in \mathcal{A}^{1,1}$ is a connection with curvature $F_{A_{0}}$ of type $(1, 1)$, and $\phi_{0}$ is holomorphic with respect to the connection $A_{0}$, i.e. ${\bar \partial}_{A_{0}}u_{0}=0$. Consider the flow  $A(t)=g(t)(A_{0}), \phi(t)=g(t)(\phi_{0}), g(t)\in \mathcal{G}^{\mathbb{C}}$. Then we have
following lemmas, whose proofs are very similar to Hong's proofs in \cite{H}.
\begin{lemma} \label{lem1}
 If the initial value $(A_{0}, \phi_{0})\in\mathcal{H}$, then the flow $(A(t), \phi(t))\in\mathcal{H}$ for any $t$, i.e.
 \begin{eqnarray}
{\bar \partial}_{A(t)}\phi(t)=0, \quad A(t)\in \mathcal{A}^{1,1}.
\end{eqnarray}
\end{lemma}
\begin{proof} By \ref{G1} and \ref{G2}, we have \begin{eqnarray*}
{\bar \partial}_{g(A_{0})}&=&g\cdot {\bar \partial}_{A_{0}}\cdot g^{-1}; \\
g^{-1}F_{g(A_{0})}g&=& F_{A_{0}}+{\bar \partial}_{A_{0}}(h^{-1}\partial h), \end{eqnarray*}
where $h=g^{*}g$.
Thus by the transformation formulas \ref{G2} and conditions above we have \begin{eqnarray*}
F_{A(t)}=F_{g^{*}(A_{0})}=g(F_{A_{0}}+{\bar \partial}_{A_{0}}(h^{-1}\partial h))g^{-1}\in \Omega^{1,1}. \end{eqnarray*}
On the other hand, by \ref{G2} and conditions above we have
\begin{eqnarray*} {\bar \partial}_{A(t)}\phi(t)=g\cdot {\bar \partial}_{A_{0}}\cdot g^{-1}g\phi_{0}=g({\bar \partial}_{A_{0}}\phi_{0})=0.
\end{eqnarray*}
This completes the proof. $\hfill\Box$
\end{proof}

\begin{lemma}  \label{lem2}
Assume the same conditions as lemma \ref{lem1}. and write\begin{eqnarray*} A=A(t)=g(t)(A_{0}), \phi=\phi(t)=g(t)(u_{0}),  \end{eqnarray*} we have \begin{eqnarray}
({\bar \partial}_{A}-\partial_{A})(\phi\otimes\phi^{*}-\tau I))= -2J_{A, \phi}. \end{eqnarray}
\end{lemma}
\begin{proof}
By lemma \ref{lem1} and the K$\ddot{\textmd{a}}$hler identities \ref{Ka2}, we have
\begin{eqnarray}
({\bar \partial}_{A}-\partial_{A})(\phi\otimes\phi^{*}-\tau I))&=&
-\partial_{A}\phi\otimes\phi^{*}+\phi\otimes({\partial}_{A}\phi)^{*}.
\end{eqnarray}
On the other hand, by lemma \ref{lem1},we have
\begin{eqnarray}
J_{A, \phi}=\frac{1}{2}(d_{A}\phi\otimes\phi^{*}-\phi\otimes(d_{A}\phi)^{*})
=\frac{1}{2}(\partial_{A}\phi\otimes\phi^{*}-\phi\otimes({\partial}_{A}\phi)^{*}).
\end{eqnarray}
We complete the proof. $\hfill\Box$
\end{proof}

\begin{lemma}  \label{lem3}
Assume the same conditions as lemma \ref{lem1}, we have
\begin{eqnarray}
d_{A}^{*}F_{A}=i(\partial_{A}-{\bar \partial}_{A})\Lambda F_{A}. \end{eqnarray}
\end{lemma}
\begin{proof} First by Bianchi identity $d_{A}F_{A}=0$ we have \begin{equation}{\partial}_{A}F_{A}=-{\bar \partial}_{A}F_{A}. \end{equation}
By lemma \ref{lem1} we know $F_{A}\in \Omega^{1,1}$, so $ {\partial}_{A}F_{A}\in \Omega^{2,1}, {\bar \partial}_{A}F_{A}\in \Omega^{1,2}$, we have \begin{equation} {\partial}_{A}F_{A}={\bar \partial}_{A}F_{A}=0. \end{equation}
Thus by the K$\ddot{\textmd{a}}$hler identities \ref{Ka1},  \ref{Ka2} we have
\begin{eqnarray*} d_{A}^{*}F_{A}&=&({\bar \partial}_{A}^{*}+{\partial}_{A}^{*})F_{A}\\
&=&i({\partial}_{A}-{\bar \partial}_{A})\Lambda F_{A}-i\Lambda({\partial}_{A}F_{A}-{\bar \partial}_{A}F_{A})\\
&=&i({\partial}_{A}-{\bar \partial}_{A})\Lambda F_{A}. \end{eqnarray*}
This completes the proof. $\hfill\Box$
\end{proof}
\begin{lemma}  \label{lem4}
Assume the same conditions as lemma \ref{lem1}, we have \begin{eqnarray*}
d_{A}^{*}d_{A}=i\Lambda({\bar \partial}_{A}{\partial}_{A}-{\partial}_{A}{\bar \partial}_{A})
\end{eqnarray*} and
\begin{eqnarray*}
d_{A}^{*}d_{A}-i\Lambda F_{A}=2{\bar \partial}_{A}^{*}{\bar \partial}_{A}.
\end{eqnarray*}
\end{lemma}
\begin{proof} By the K$\ddot{\textmd{a}}$hler identities and lemma \ref{lem1}, \begin{eqnarray*}
d^{*}_{A}d_{A}&=&({\bar \partial}^{*}_{A}+{\partial}^{*}_{A})({\bar \partial}_{A}+{\partial}_{A})\\
&=&i\Lambda({\bar \partial}_{A}{\partial}_{A}-{\partial}_{A}{\bar \partial}_{A}),
\end{eqnarray*}
and we also have \begin{eqnarray*} d^{*}_{A}d_{A}&=&{\partial}^{*}_{A}{\partial}_{A}+{\bar \partial}^{*}_{A}{\bar \partial}_{A}. \end{eqnarray*}
Then
\begin{eqnarray*}
{\partial}^{*}_{A}{\partial}_{A}-{\bar \partial}^{*}_{A}{\bar \partial}_{A}&=&i\Lambda{\bar \partial}_{A}{\partial}_{A}+i\Lambda{\partial}_{A}{\bar \partial}_{A}\\
&=&i\Lambda(d^{2}_{A})\\
&=&i\Lambda F_{A}. \end{eqnarray*}
Therefore we have \begin{eqnarray*}d_{A}^{*}d_{A}-i\Lambda F_{A}=2{\bar \partial}_{A}^{*}{\bar \partial}_{A}. \end{eqnarray*} This completes the proof. $\hfill\Box$
 \end{proof}
Now we will prove the following equality, which is key to later estimates.
\begin{lemma} \label{equ}
For the curvature $F_{A}$ of the connection $A$, denote
$$\hat{e}:=|F_{A, \phi}|^{2}=|\Lambda F_{A}-\frac{i}{2}(\phi\otimes\phi^{*}-\tau I)|^{2}.$$
Let $(A, \phi)(t, x)$ be a solution to the gradient flow equations \ref{f2} on $X\times[0, T)$ with $T\leq +\infty$. Then we have
\begin{eqnarray*}
 (\frac{\partial}{\partial t}+\Delta)\hat{e}=-2|\nabla_{A}(\Lambda F_{A}-\frac{i}{2}\phi\otimes\phi^{*})|^{2}-2|\frac{i}{2}(|\phi|^{2}-\tau )\phi-\Lambda F_{A}\phi|^{2}.
\end{eqnarray*}
where $\nabla_{A}=d_{A}$, and $\Delta=\nabla_{A}^{*}\nabla_{A}$ is the Laplacian.
\end{lemma}
\begin{proof}
By direct computation, we have \begin{eqnarray*}
\frac{1}{2}\frac{\partial}{\partial t}\hat{e}&=&Re\langle\frac{\partial}{\partial t}(\Lambda F_{A})+\frac{i}{2}\frac{\partial}{\partial t}(\phi\otimes\phi^{*}), \Lambda F_{A}-\frac{i}{2}(\phi\otimes\phi^{*}-\tau I)\rangle\\
&=& Re\langle\frac{\partial}{\partial t}(\Lambda F_{A}-\frac{i}{2}[\frac{\partial}{\partial t}\phi\otimes\phi^{*}+\phi\otimes(\frac{\partial}{\partial t}\phi)^{*}], \Lambda F_{A}-\frac{i}{2}(\phi\otimes\phi^{*}-\tau I)\rangle.
\end{eqnarray*}
Now by the gradient flow equations \ref{f2}, lemma \ref{lem3} and lemma \ref{lem4} we get
\begin{eqnarray*}
&&Re\langle\frac{\partial}{\partial t}(\Lambda F_{A}), \Lambda F_{A}+\frac{i}{2}(\phi\otimes\phi^{*}-\tau I)\rangle\\
&=&Re\langle\Lambda d_{A}(\frac{\partial A}{\partial t}), \Lambda F_{A}+\frac{i}{2}(\phi\otimes\phi^{*}-\tau I)\rangle\\
&=&Re\langle-\Lambda d_{A}(d_{A}^{*}F_{A}+J_{A, \phi}),  \Lambda F_{A}+\frac{i}{2}(\phi\otimes\phi^{*}-\tau I)\rangle\\
&=&Re\langle i\Lambda d_{A}({\bar \partial}_{A}-{\partial}_{A})\Lambda F_{A}-\Lambda d_{A}J_{A, \phi}, \Lambda F_{A}+\frac{i}{2}(\phi\otimes\phi^{*}-\tau I)\rangle \\
&=&-Re\langle\nabla_{A}^{*}\nabla_{A}(\Lambda F_{A})+\Lambda d_{A}J_{A, \phi},  \Lambda F_{A}+\frac{i}{2}(\phi\otimes\phi^{*}-\tau I)\rangle,
\end{eqnarray*}
and
\begin{eqnarray*}
&&Re\langle-\frac{i}{2}[\frac{\partial}{\partial t}\phi\otimes\phi^{*}+\phi\otimes(\frac{\partial}{\partial t}\phi)^{*}], \Lambda F_{A}-\frac{i}{2}(\phi\otimes\phi^{*}-\tau I)\rangle\\
&=&Re\langle-\frac{i}{2}[-\nabla^{*}_{A}\nabla_{A}\phi+\frac{1}{2}\phi(\tau-|\phi|^{2})]\otimes \phi^{*}, \Lambda F_{A}-\frac{i}{2}(\phi\otimes\phi^{*}-\tau I)\rangle\\
&&+Re\langle-\frac{i}{2}\phi\otimes[-\nabla^{*}_{A}\nabla_{A}\phi+\frac{1}{2}\phi(\tau-|\phi|^{2})]^{*}, \Lambda F_{A}-\frac{i}{2}(\phi\otimes\phi^{*}-\tau I)\rangle\\
&=&-\frac{1}{2}Re\langle-i\nabla^{*}_{A}\nabla_{A}\phi \otimes\phi^{*}-i\phi\otimes (\nabla^{*}_{A}\nabla_{A}\phi^{*}), \Lambda F_{A}-\frac{i}{2}(\phi\otimes\phi^{*}-\tau I)\rangle\\
&&+\frac{1}{2}Re\langle  i(|\phi|^{2}-\tau)\phi\otimes\phi^{*}, \Lambda F_{A}-\frac{i}{2}(\phi\otimes\phi^{*}-\tau I) \rangle\\
&=&-\frac{1}{2}Re\langle-i\nabla^{*}_{A}\nabla_{A}\phi \otimes\phi^{*}-i\phi\otimes (\nabla^{*}_{A}\nabla_{A}\phi^{*}), \Lambda F_{A}-\frac{i}{2}(\phi\otimes\phi^{*}-\tau I)\rangle\\
&&+\frac{1}{2}Re\langle  i(|\phi|^{2}-\tau)\phi, \Lambda F_{A}\phi\rangle-\frac{1}{4}|\phi|^{2}(\tau-|\phi|^{2})^{2}.
\end{eqnarray*}
Then add equalities above, we have \begin{eqnarray*}
\frac{1}{2}\frac{\partial}{\partial t}\hat{e}&=&-Re\langle\nabla_{A}^{*}\nabla_{A}(\Lambda F_{A})+\Lambda d_{A}J_{A, \phi},  \Lambda F_{A}+\frac{i}{2}(\phi\otimes\phi^{*}-\tau I)\rangle\\
&&-\frac{1}{2}Re\langle-i\nabla^{*}_{A}\nabla_{A}\phi \otimes\phi^{*}-i\phi\otimes (\nabla^{*}_{A}\nabla_{A}\phi^{*}), \Lambda F_{A}-\frac{i}{2}(\phi\otimes\phi^{*}-\tau I)\rangle\\
&&+\frac{1}{2}Re\langle  i(|\phi|^{2}-\tau)\phi, \Lambda F_{A}\phi\rangle-\frac{1}{4}|\phi|^{2}(\tau-|\phi|^{2})^{2}.
\end{eqnarray*}
On the other hand, we have\begin{eqnarray*}
\frac{1}{2}\triangle \hat{e}&=&\nabla^{*}_{A}Re\langle\nabla_{A}(\Lambda F_{A}-\frac{i}{2}(\phi\otimes\phi^{*}-\tau I), \Lambda F_{A}-\frac{i}{2}(\phi\otimes\phi^{*}-\tau I)\rangle\\
&=&Re\langle\nabla^{*}_{A}\nabla_{A}(\Lambda F_{A}-\frac{i}{2}(\phi\otimes\phi^{*}-\tau I)), \Lambda F_{A}-\frac{i}{2}(\phi\otimes\phi^{*}-\tau I)\rangle\\
&&-|\nabla_{A}(\Lambda F_{A}-\frac{i}{2}(\phi\otimes\phi^{*}-\tau I))|^{2}\\
&=&Re\langle\nabla_{A}^{*}\nabla_{A}(\Lambda F_{A})+\Lambda d_{A}J_{A, \phi},  \Lambda F_{A}+\frac{i}{2}(\phi\otimes\phi^{*}-\tau I)\rangle-|\nabla_{A}(\Lambda F_{A}-\frac{i}{2}(\phi\otimes\phi^{*}-\tau I))|^{2}\\
&&+\frac{1}{2}Re\langle-i\nabla^{*}_{A}\nabla_{A}\phi \otimes\phi^{*}-i\phi\otimes (\nabla^{*}_{A}\nabla_{A}\phi^{*}), \Lambda F_{A}-\frac{i}{2}(\phi\otimes\phi^{*}-\tau I)\rangle\\
&&-\frac{1}{2}Re\langle F_{A}\phi\otimes\phi^{*}-\phi^{*}\otimes(F_{A}\phi)*, \Lambda F_{A}-\frac{i}{2}(\phi\otimes\phi^{*}-\tau I)\rangle \\
&=&Re\langle\nabla_{A}^{*}\nabla_{A}(\Lambda F_{A})+\Lambda d_{A}J_{A, \phi},  \Lambda F_{A}+\frac{i}{2}(\phi\otimes\phi^{*}-\tau I)\rangle-|\nabla_{A}(\Lambda F_{A}-\frac{i}{2}(\phi\otimes\phi^{*}-\tau I))|^{2}\\
&&+\frac{1}{2}Re\langle-i\nabla^{*}_{A}\nabla_{A}\phi \otimes\phi^{*}-i\phi\otimes (\nabla^{*}_{A}\nabla_{A}\phi^{*}), \Lambda F_{A}-\frac{i}{2}(\phi\otimes\phi^{*}-\tau I)\rangle\\
&&+\frac{1}{2}Re\langle  i(|\phi|^{2}-\tau)\phi, \Lambda F_{A}\phi\rangle-|\Lambda F_{A}\phi|^{2}.
\end{eqnarray*}
Combining equalities above we get
\begin{eqnarray*}
(\frac{\partial}{\partial t}+\Delta)\hat{e}&=&-2|\nabla_{A}(\Lambda F_{A}-\frac{i}{2}(\phi\otimes\phi^{*}-\tau I))|^{2}\\
&&-2|\Lambda F_{A}\phi|^{2}-\frac{1}{2}|\phi|^{2}(\tau-|\phi|^{2})^{2}+2Re\langle  i(|\phi|^{2}-\tau)\phi, \Lambda F_{A}\phi\rangle\\
&=&-2|\nabla_{A}(\Lambda F_{A}-\frac{i}{2}(\phi\otimes\phi^{*}-\tau I))|^{2}-2|\frac{i}{2}(|\phi|^{2}-\tau)\phi-\Lambda F_{A}\phi|^{2}.
\end{eqnarray*}
This completes the proof.$\hfill\Box$
 \end{proof}
\begin{corollary}\label{cor1}
 Assume the same conditions as in lemma \ref{equ}, then
$\sup_{X}\hat{e}$ is a decreasing function of $t$.
\end{corollary}
\begin{proof}
By lemma \ref{equ} we have
$$
(\frac{\partial}{\partial t}+\Delta)\hat{e}\leq0,$$
by the maximum principal for the heat operator $(\frac{\partial}{\partial t})+\triangle$ on $X$ we prove our claim. $\hfill\Box$
\end{proof}
Essentially the following result can be found as Lemma 4 in \cite{H}.
\begin{lemma}\label{lem5}
Let $(A, \phi)(t, x)$ be a solution to the gradient flow equations on $X\times[0, T)$ with $T\leq +\infty$. Then for all $t\in [0, T)$, $$|\phi(t, x)|\leq \max\{\sup_{X}|\phi_{0}|, \tau\}.$$
\end{lemma}
\begin{proof}
By the gradient flow equations \ref{f2}, we have
\begin{eqnarray*}
(\frac{\partial}{\partial t}+\triangle)|\phi|^{2}&=&2Re\langle\frac{\partial\phi}{\partial t}, \phi\rangle+2Re\langle\nabla_{A}^{*}\nabla_{A}\phi, \phi\rangle-2|\nabla_{A}\phi|^{2}\\
&=&|\phi|^{2}(\tau-|\phi|^{2})-2|\nabla_{A}\phi|^{2}.
\end{eqnarray*}
Assume that $|\phi|^{2}$ attains its maximum on $X\times[0, T)$ at the point $(x_{0}, t_{0})$ with $0<t_{0}<T$. If $|\phi|^{2}(t_{0}, x_{0})\leq\tau$, it completes the proof. Otherwise, if $|\phi|^{2}(t_{0}, x_{0})>\tau$, then
$$(\frac{\partial}{\partial t}+\triangle)|\phi|^{2}(t_{0}, x_{0})<0. $$
This is contradicted with the maximum principle of the heat operator $(\frac{\partial}{\partial t}+\triangle)$. Then $|\phi|^{2}$ must attain its maximum point $\sup_{X}|\phi_{0}|$. This completes our proof.
$\hfill\Box$
\end{proof}
By corollary \ref{cor1} and lemma \ref{lem5}, we immediately get
\begin{lemma}\label{lem6}
Let $(A, \phi)(t, x)$ be a solution to the gradient flow equations \ref{f2} on $X\times[0, T)$ with $T\leq +\infty$, and denote
$$\hat{F}=\Lambda F_{A}.$$
Then $\sup_{X}|\hat{F}|$ is uniformly bounded for $t\in [0, T)$.
\end{lemma}

\section{Existence}
\subsection{Local existence}
In this subsection we prove the local existence of the gradient flow equations \ref{f2}.
For the proof, we adopt Donaldson's approach \cite{D, H} to consider a flow of gauge transformation:
\begin{equation}\label{H}
\frac{\partial h}{\partial t}=-2ih[\Lambda F_{A_{0}}+\Lambda\bar{\partial}_{A_{0}}(h^{-1}\partial_{A_{0}}h)-\frac{i}{2}(\phi_{0}\otimes\phi_{0}^{*}h-\tau I))],
\end{equation}
with $h(0)=I$.
 \begin{theorem}
Assume $(A_{0}, \phi_{0})\in \mathcal{H}$, then there exist a positive constant $\epsilon>0$ and a smooth solution $(A(x, t), \phi(x, t))$ such that $(A(x, t), \phi(x, t))$ solve the gradient flow equations \ref{f2}
 \begin{eqnarray*}
\frac{\partial A}{\partial t}&=&-d^{*}_{A}F_{A}-J_{A,\phi};\\
\frac{\partial \phi}{\partial t}&=&-d^{*}_{A}d_{A}\phi+\frac{1}{2}\phi(\tau-|\phi|^{2}).
\end{eqnarray*}
in $X \times [0, \epsilon)$ with initial values $A(x, 0)=A_{0}$ and $u(x, 0)=u_{0}$.
  \end{theorem}

Recall the complex gauge group ${\mathcal{G}}^{\mathbb{C}}$, which acts on $A\in \mathcal{A}^{1,1}$ with curvature $F_{A}$ of type $(1,1)$ by
\begin{eqnarray*}
\bar{\partial}_{g(A)}=g\cdot \bar{\partial}_{A}\cdot g^{-1}, \quad {\partial}_{g(A)}={g^{*}}^{-1} \cdot {\partial}_{A}\cdot g^{*}
\end{eqnarray*}
where $g^{*}=\bar{g}^{t}$ denotes the conjugate transpose of $g$. Extending the action of the unitary gauge group,
\begin{eqnarray*}
\mathcal{G}=\{g\in {\mathcal{G}}^{\mathbb{C}}|h(g)=g^{*}g=I\},\end{eqnarray*}
which means
\begin{eqnarray*}
g^{-1}\cdot d_{g(A)}\cdot g=\bar{\partial}_{A}+h^{-1}{\partial}_{A}h,\quad
g^{-1}F_{g(A)}g=F_{A}+\bar{\partial}_{A}(h^{-1}{\partial}_{A}h),
\end{eqnarray*}
where $h=g^{*}g$.\par
Consider the following heat equation for a one-parameter family $H=H(t)$ of metrics on the holomorphic vector bundle $E$ over $X$:
\begin{equation}\label{H}
\frac{\partial H}{\partial t}=-2iH[\Lambda{F}_{H}-\frac{i}{2}(\phi\otimes\phi^{*}-\tau I)], \end{equation}
Put $H(t)=H_{0}h$, where $H_{0}$ is the initial Hermitian metric on $E$, then
the equation \ref{H} is completely equivalent to the equation of gauge transformation:
\begin{equation}\label{h1}
\frac{\partial h}{\partial t}=-2ih[\Lambda{F}_{H}-\frac{i}{2}(\phi\otimes\phi^{*}-\tau I)], \end{equation}
with initial value $$h(0)=I$$
where $\Lambda{F}_{H}=\Lambda{F}_{A_{0}}+\Lambda\bar{\partial}_{A_{0}}(h^{-1}\partial_{A_{0}}h)$.\par
We can rewrite this equation as

 \begin{equation}\label{h2}
\frac{\partial h}{\partial t}=-\Delta_{A_{0}}h-i(h\Lambda{F}_{A_{0}}+\Lambda{F}_{A_{0}}h)+2i\Lambda({\bar \partial}_{A_{0}}h h^{-1}{\partial}_{A_{0}})-[\frac{1}{4}[(h+h^{*})\phi_{0}]\otimes[(h+h^{*})\phi_{0}]^{*}-\tau h], \end{equation}
with $h(0)=I$ if $h^{*}=h$, where $\Delta_{A_{0}}=2\bar \partial_{A_{0}}^{*}\partial_{A_{0}}$ is the Laplacian defined through an integrable connection $A_{0}$ $({\bar \partial}^{2}_{A_{0}}=0)$.
As pointed out in \cite{D, J}, (\ref{h2}) is a nonlinear parabolic equation, so we obtain a short time solution to (\ref{h1}) by standard parabolic PDE theory. Therefore we obtain a short-time solution to (\ref{H}).\par
\begin{lemma}\label{local1}
Given a Hermitian $H_{0}$ of class $C^{2, \alpha}$ on $E$, there exists a constant $\varepsilon>0$ with such that the equation (\ref{h1}) (equivalently \ref{H}) exists a solution with initial value $h(0)=I$ exists for $0\leq t<\varepsilon$. Further the solution is of class $C^{2, \alpha}$ and $\frac{\partial H}{\partial t}$ is of class $C^{\alpha}$.
\end{lemma}

Suppose $h=h(t)$ is a local solution of the equation \ref{h1} on $X\times [0, \varepsilon)$. Take any $g\in \mathcal{G}^{\mathbb{C}}$ such that $\bar{g}^{t}g=h$ (for example $g=h^{1/2}$). Since $h$ solve (\ref{h1}), we have
\begin{equation} \label{G3}
\frac{\partial g}{\partial t}g^{-1}+\bar{g}^{t^{-1}}\frac{\partial \bar{g}^{t}}{\partial t}=-2i[\Lambda F_{g(A_{0})}-\frac{i}{2}(\phi\otimes\phi^{*}-\tau I)].
\end{equation}

Here we use a fact that the moment map $\mu$ is equivariant, i.e., $g\mu(u_{0})g^{-1}=\mu(g(u_{0}))=\mu(u)$.
Then as in \cite{H}, at the time $t$,
\begin{eqnarray*} \frac{\partial A}{\partial t}&=&\frac{\partial}{\partial\epsilon}|_{\epsilon=0}(\partial_{(g+\epsilon\partial_{t}g)A_{0}}+{\bar\partial}_{(g+\epsilon\partial_{t}g)A_{0}})\\
&=&\partial_{A}(\bar{g}^{t^{-1}}\partial_{t}\bar{g}^{t})+{\bar\partial}_{A}(g\partial_{t}g^{-1}).\end{eqnarray*}
So we have
\begin{eqnarray*}
\frac{\partial A}{\partial t}&=&\partial_{A}(\bar{g}^{t^{-1}}\partial_{t}\bar{g}^{t})-{\bar \partial}_{A}(\partial_{t}g g^{-1})\\
&=&-{ \partial}_{A}i[\Lambda F_{A}-\frac{i}{2}(\phi\otimes\phi^{*}-\tau I)]-\frac{1}{2}{\partial}_{A}(\partial_{t}g g^{-1})+\frac{1}{2}{\partial}_{A}(\bar{g}^{t^{-1}}\partial_{t}\bar{g}^{t})+\\
& & {\bar \partial}_{A}i[\Lambda F_{A}-\frac{i}{2}(\phi\otimes\phi^{*}-\tau I)]+\frac{1}{2}{\bar \partial}_{A}(\bar{g}^{t^{-1}}\partial_{t}\bar{g}^{t})-\frac{1}{2}{\bar \partial}_{A}(\partial_{t}g g^{-1}) \\
 & =& i({\bar \partial}_{A}-{\partial}_{A})[\Lambda F_{A}-\frac{i}{2}(\phi\otimes\phi^{*}-\tau I)]+d_{A}(\alpha(t))\\
& =&-i(\partial_{A}-{\bar \partial}_{A})\Lambda F_{A}+i({\bar \partial}_{A}-\partial_{A})[-\frac{i}{2}(\phi\otimes\phi^{*}-\tau I)]+d_{A}(\alpha(t)).\\
 \end{eqnarray*}
Due to lemma \ref{lem2} and lemma \ref{lem3}, we have
 \begin{eqnarray}\label{At}
 \frac{\partial A}{\partial t}&=&-d_{A}^{*}F_{A}-J_{A, \phi}+d_{A}(\alpha(t)),
 \end{eqnarray}
 where $\alpha(t)$ is defined by $\alpha(t)=\frac{1}{2}(\bar{g}^{t^{-1}}\partial_{t}\bar{g}^{t}-\partial_{t}g g^{-1})$.\\

 Applying \ref{G3} we obtain
 \begin{eqnarray*}
\frac{\partial \phi}{\partial t}& = & \partial_{t}g g^{-1}\phi\\
 & =& -i[\Lambda F_{A}-\frac{i}{2}(\phi\otimes\phi^{*}-\tau I)]\phi-\frac{1}{2}(\bar{g}^{t^{-1}}\partial_{t}\bar{g}^{t}-\partial_{t}g g^{-1})\phi\\
 & =&-i\Lambda F_{A}\phi+\frac{1}{2}\phi(\tau-|\phi|^{2})-\alpha(t)\phi.
\end{eqnarray*}
By lemma \ref{lem1} and lemma \ref{lem4}, we have
\begin{eqnarray} \label{fit}
\frac{\partial \phi}{\partial t}& =&-d_{A}^{*}d_{A}\phi+\frac{1}{2}\phi(\tau-|\phi|^{2})-\alpha(t)\phi.
\end{eqnarray}

Assume $h$ is the solution of \ref{h2}, let $g=h^{\frac{1}{2}}$. The corresponding pair $(\tilde{A}(t), \tilde{u}(t))=(g(A_{0}), g(u_{0}))$ thus is a solution to \ref{At} and \ref{fit}. \par
At last we can prove the local existence of the gradient flow equations \ref{f2} through a gauge transformation of the equivalent flow \ref{At} and \ref{fit}.
Let $S(t)$ be the unique smooth solution to the following initial value problem:
\begin{eqnarray}\label{St}
\frac{d S}{dt}\cdot S^{-1}=- \alpha(t), \quad S(0)=I,
 \end{eqnarray}
where \begin{eqnarray*}
\alpha(t)=\frac{1}{2}(\bar{g}^{t^{-1}}\partial_{t}\bar{g}^{t}-\partial_{t}g g^{-1}). \end{eqnarray*}
Let \begin{eqnarray*}d_{A}=S^{-1}\cdot d_{\tilde{A}}\cdot S, \quad \phi=S^{-1}\tilde{\phi}, \end{eqnarray*}
we have \begin{eqnarray*}
S^{-1}\cdot{d^{*}_{\tilde A}F_{\tilde A}}\cdot S&=&d^{*}_{A}F_{A}; \\
S^{-1}\cdot (\tilde{d}_{A}\tilde{\phi}\otimes {\tilde{\phi}}^{*}-{\tilde{\phi}}\otimes (\tilde{d}_{A}\tilde{\phi})^{*})\cdot S&=& d_{A}\phi\otimes \phi^{*}-\phi\otimes(d_{A}\phi)^{*},
\end{eqnarray*}
and \begin{eqnarray*}
d_{\tilde A}(\alpha)=d_{\tilde A}\cdot \alpha-\alpha\cdot d_{\tilde A}.\end{eqnarray*}
Combining these equalities with \ref{St}, we have the following local existence theorem.
\begin{theorem}\label{loc2}
For any initial condition $(A_{0}, \phi_{0})\in \mathcal{H}$, which means that $A_{0}\in \mathcal{A}^{1,1}$ is a given smooth connection on $E$ with curvature $F_{A_{0}}$ of type $(1,1)$ and $\phi_{0}$ is a holomorphic section of $E$, i.e. ${\bar \partial}_{A_{0}}\phi_{0}=0$. Then there exists a positive constant $\epsilon >0$ such that, the gradient flow equations \ref{f2}
\begin{eqnarray*}
\frac{\partial A}{\partial t}&=&-d^{*}_{A}F_{A}-J_{A,\phi};\\
\frac{\partial \phi}{\partial t}&=&-d^{*}_{A}d_{A}\phi+\frac{1}{2}\phi(\tau-|\phi|^{2}).
\end{eqnarray*}
of the vortex functional \ref{vort} have a local smooth solution on $X \times [0, \epsilon)$.
\end{theorem}

\subsection{Global existence}
Let $H=H_{0}h$ and $h=\bar{g}^{t}g$, and put
\begin{eqnarray*}
\hat{F}_{H}:=\Lambda F_{A_{0}}+\Lambda\bar{\partial}_{H_{0}}(h^{-1}\partial_{H_{0}}h),
\end{eqnarray*}
Then equation \ref{h2} can be rewritten as the following
\begin{eqnarray}\label{h3}
\frac{\partial h}{\partial t}=-2i\bar{g}^{t}\big[\Lambda F_{A}-\frac{i}{2}(\phi\otimes\phi^{*}-\tau I)\big]g,
\end{eqnarray}
with initial value $h(0)=I$, where $\bar{g}^{t}g=h$, $A=g(A_{0})$, $\phi=g(\phi_{0})$.\par
By lemma \ref{cor1} and the Gr$\ddot{\textmd{o}}$nwall's inequality, we can prove the following $C^{0}$ convergence result totally similar to the proof of Lemma 8 in \cite{H}.
\begin{lemma} \label{lem8}
Let $H(t)=H_{0}h(t)$ where $h(t)$ is smooth solution of (\ref{h3}) in $X\times [0, T)$ with the initial value $H_{0}$ for a finite time $T>0$. Then $H(t)$ converges in $C^{0}$ to a nondegenerate continuous metric $H(T)$ as $t\rightarrow T$.
\end{lemma}
Similar to Lemma 9 in \cite{H}, we have
\begin{lemma}\label{uniqueness}  Any two smooth solutions $H=H_{0}h, K=H_{0}k$ of the equation \ref{h3}, which are defined for $0\leq t< \epsilon$, are continuous at $t=0$, and we have the same initial condition $H_{0}=K_{0}$, agree for all $t\in [0, \epsilon)$.
\end{lemma}
\begin{proof}
By the equation \ref{h2}, we have
\begin{eqnarray*}
\frac{\partial(h-k)}{\partial t}&=&-\Delta_{A_{0}}(h-k)-2i(h-k)\Lambda F_{A_{0}}\\
&&+2i\Lambda \big[{\bar \partial}_{A_{0}}hh^{-1}{\partial}_{A_{0}}h-{\bar \partial}_{A_{0}}kk^{-1}{\partial}_{A_{0}}k\big]\\
&&-(h\phi_{0}\otimes\phi_{0}^{*}h-k\phi_{0}\otimes\phi_{0}^{*}k)+\tau(h-k),
\end{eqnarray*}
where $\Delta_{A_{0}}=2\partial_{A_{0}}^{*}\partial_{A_{0}}$. \par
Then we can multiply both sides with $h-k$ and integrate by parts (see \cite{H} for details) to show that
$$f(\tau)\leq C\tau f(\tau),$$
where $f(\tau):=\sup_{0\leq t\leq\tau}\int_{X}|h-k|^{2}dX$. Choosing $C\tau<1$ we obtain $f(\tau)=0$. This means that $h(t)=k(t)$ for all $t\leq\tau$. $\hfill\Box$
\end{proof}

Using almost the same procedure as in \cite{D, H}, we can further prove
\begin{lemma} \label{lem9}
Let $H(t), 0\leq t< T$, be any one-parameter family of Hermitian metrics on a holomorphic vector bundle $P$ over the compact K$\ddot{\textmd{a}}$hler manifold $X$ such that\\
(i) $H(t)$ converges in $C^{0}$ to some continuous metric $H(T)$ as $t\rightarrow T$,\\
(ii) ${sup}_{X}|{\hat{F}}_{H}|$ is uniformly bounded for $t< T$.\\
Then $H(t)$ is bounded in $C^{1,\alpha}$ independently of $t\in[0, T)$ for $0<\alpha<1$.
\end{lemma}
Now we can prove the global existence of (\ref{H}) as follows\\
\begin{theorem} \label{Global}
 For any initial condition $(A_{0}, \phi_{0})\in \mathcal{H}$, which means that $A_{0}\in \mathcal{A}^{1,1}$ is a given smooth connection on $E$ with curvature $F_{A_{0}}$ of type $(1,1)$ and $\phi_{0}$ is a holomorphic section of $E$, i.e. ${\bar \partial}_{A_{0}}\phi_{0}=0$, and let $H_{0}$ be a Hermitian metric on $E$. Then the equation
\begin{eqnarray*}
\frac{\partial H}{\partial t}=-2iH\big[\Lambda F_{H}-\frac{i}{2}(\phi\otimes\phi^{*}-\tau I)\big]
\end{eqnarray*}
has a unique solution $H(t)$ which exists on $X$ for $0\leq t<\infty$.
\end{theorem}
\begin{proof} The local existence and uniqueness of this equation has been proved by lemma \ref{local1} and lemma \ref{uniqueness}. Suppose that the solution exists for $0\leq t< T$. By
 lemma \ref{lem6}, $\sup_{X}|\hat{F}|$ is uniformly bounded for $t\in [0, T)$. By lemma \ref{lem8}, $H(t)$ converges in $C^{0}$ to a nondegenerate continuous limit metric $H(T)$ as $t\rightarrow T$. Thus by lemma \ref{lem9}, $H(t)$ is bounded in $C^{1,\alpha}$ independently of $t\in [0, T)$.
  Let $H(t)=H_{0}h(t)$, where $h(t)$ solves
 \begin{eqnarray*}\frac{\partial h}{\partial t}=-\Delta_{A_{0}}h-2ih\Lambda{F}_{A_{0}}+2i\Lambda({\bar \partial}_{A_{0}}h h^{-1}{\partial}_{A_{0}}h)-(h\phi_{0}\otimes\phi_{0}^{*}h-\tau h). \end{eqnarray*}
Then it is the same as the proof of Theorem 11 in \cite{H} to prove that $h$ is $C^{2,\alpha}$ and $\frac{\partial h}{\partial t}$ is $C^{\alpha}$ with bounds independent of $t\in [0, T)$. Thus $H(t)\rightarrow H(T)$ in $C^{2,\alpha}$, hence in $C^{\infty}$. For the initial value $H(T)$ we use local existence theorem (\ref{local1}) again. Therefore the solution continues to exist for $t<T+\epsilon$, for some $\epsilon$. This completes the proof. $\hfill\Box$
\end{proof}

Theorem \ref{Com} follows from Theorem \ref{loc2} and Theorem \ref{Global}. Then by Theorem \ref{equiv}, we ultimately complete the proof of Theorem \ref{Exi}.

College of Science, National University of Defense Technology, Changsha 410073, Hunan, P. R. China\\
Email address: linaijin@nudt.edu.cn\\
Department of Mathematics, The University of British Columbia, Vancouver, B.C., Canada V6T 1Z2\\
Email address: lmshen@math.ubc.ca
\end{document}